\documentclass[a4paper]{article}
\usepackage[affil-it]{authblk}                    
\usepackage{graphicx}
\usepackage{cite}
\usepackage{amsfonts}
\usepackage{amsmath}
\usepackage{amsthm}
\usepackage[english]{babel}
\newtheorem{theorem}{Theorem}
\newtheorem{corollary}{Corollary}[theorem]
\newtheorem{remark}{Remark}[section]
\begin{document}

\title{Optimality conditions for an exhausterable function on an exhausterable set
}


\author{Majid E. Abbasov         
}


\affil{
              St. Petersburg State University, SPbSU,\\ 7/9 Universitetskaya nab., St.
              Petersburg, 199034 Russia. \\
              {m.abbasov@spbu.ru, abbasov.majid@gmail.com}           
}


\maketitle

\begin{abstract}
Exhausters are families of convex compact sets that allow one to
represent directional derivative of the studied function at the
considered point in the form of InfMax or SupMin of linear
functions. Functions for which such a representation is valid we
call exhausterable. The class of these functions is quite wide and
contains many nonsmooth ones. The set which is given by
exhausterable function is also called exhausterable.

In the present paper we describe optimality conditions for an
exhausterable function on an exhausterable set. These conditions
can be used for solving of many nondifferentiable optimization
problems. An example that illustrate obtained results is provided.

\end{abstract}

\section*{Introduction}
\label{intro_opt_cond_for_exh_func_exh_set}

Directionally differentiable functions form a wide and important
class of nonsmooth functions. It includes convex functions,
maximum and minimum functions and others.

Subdifferential notion \cite{rock70} can be used effectively for
working with directional derivative and therefore for solution of
convex nonsmoth optimization problems. Researchers tried to
develop approaches also for nonconvex problems. The most known
invented tools are subdifferentials of Clarke
\cite{Clarke_83,Hare13}, Mordukhovich \cite{Mordukhovich_80},
Michel-Penot \cite{Michel-Penot_84,Hiriart-Urruty_99} and others.
The concept of exhausters holds a special place in this series due
to its constructiveness.

The idea of exhausters notion goes back to the works of
Pshenichny, Rubinov and Demyanov. Pshenichny in
\cite{Demyanov-psch80} introduced the definition of upper convex
approximation. Demyanov and Rubinov
\cite{Demyanov-dr82,Demyanov-Rubinov2001} proposed to consider
exhaustive families of upper convex and lower concave
approximations, and then introduced exhauster notion
\cite{Demyanov_optimization_99,Demyanov-dem00a}.

Exhausters are families of convex compact sets that allow one to
represent the directional derivative of the studied function in
the form of InfMax or SupMin of linear functions. Therefore they
provide the same representation for the approximation of a studied
function in the neighbourhood of a considered point. Functions for
which such an expansion is valid we call exhausterable.

Calculus of exhausters was described in \cite{Demyanov-dem00a}.
Formulas of this calculus allows one to build exhausters for a
wide class of functions. Unconstrained optimality conditions in
terms of these families were derived in
\cite{Demyanov-dem00a,Demyanov-dros06,Demyanov-Abbasov_jogo13,Demyanov-Abbasov_IMMO10}.
Constrained optimality conditions for an exhausterable function on
an abstractive theoretical cone were obtained in
\cite{Demyanov-dros05}.

In the present work we study constrained optimality conditions for
exhausterable function on a set which is defined via another
exhausterable function. We get new optimality conditions in terms
of exhausters of these two functions.


The paper is organized as follows. In Section 1 we discuss
directional derivative and exhausters notions. In Section 2 we
consider the statement of the problem and conic approximations of
a feasible set. Optimality conditions for an exhausterable
function on an exhausterable set are given in Section 3. An
illustrative example is provided in Section 4.

\section{Dini and Hadamard directional derivatives. Exhausters}

Let a function $f\colon \mathbb{R}^n \rightarrow  \mathbb{R}$ be
given. The function $f$ is called Dini-differentiable at a point
$x \in \mathbb{R}^n$ in a direction $g \in \mathbb{R}^{n}$, if
there exists the finite limit
\begin{equation*}\label{abbasov_eq13appr}
f^{\prime}_{D}(x,g) = \lim_{\alpha\downarrow 0} \frac{{f(x+\alpha
g)-f(x)}}{\alpha}.
\end{equation*}

The function $f$ is called Hadamard-differentiable at a point $x
\in \mathbb{R}^n$ in a direction $g \in \mathbb{R}^{n}$, if there
exists the finite limit
\begin{equation*}\label{abbsov_eq14appr}
f^{\prime}_{H}(x,g) = \lim_{[\alpha,g^{\prime}] \to
[+0,g]}\frac{f(x+\alpha g^{\prime})-f(x)}{\alpha} .
\end{equation*}

The value $f^{\prime}_{D}(x,g)$ is called the Dini derivative of
the function $f$ at the point $x\in \mathbb{R}^n$ in the direction
$g\in \mathbb{R}^n$, and the value $f^{\prime}_{H}(x,g)$ is called
the Hadamard derivative of the function $f$ at the point $x\in
\mathbb{R}^n$ in the direction $g\in \mathbb{R}^n$. The functions
$f^{\prime}_{D}(x,g) $ and $ f^{\prime}_{H}(x,g)$ are positively
homogeneous (p.h.) as functions of direction $g\in \mathbb{R}^n$.
It is clear that differentiability in the sense of Hadamard
implies differentiability in the sense of Dini. The converse is
not true.

Let $f\colon \mathbb{R}^n\to \mathbb{R}$ be a directionally
differentiable (in the sense of Dini or Hadamard) function and
$h(g)=f^{\prime}(x,g)$ be the corresponding derivative of the
function $f$ at a point $x$ in a direction $g$. Fix $x\in
\mathbb{R}^n$. In case when $h(g)$ is upper semicontinuous as a
function of $g$ it can be written in the form (see
\cite{Demyanov-dr82})
\begin{equation}\label{eq1.2}
h(g)=\inf_{C\in E^\ast} \ \max_{v\in C}\langle v,g\rangle,
\end{equation}
where $E^\ast=E^\ast(x)$ is a family of convex, closed and bounded
sets from $\mathbb{R}^n$.

If $h(g)=f^{\prime}(x,g)$ is lower semicontinuous as a function of
$g$, then it can be expressed as
\begin{equation}\label{eq1.3}
h(g)=\sup_{C\in E_\ast} \ \min_{w\in C} \langle w,g\rangle,
\end{equation}
where $E_\ast=E_\ast(x)$ is a family of convex, closed and bounded
sets from $\mathbb{R}^n$.

The family of sets $E^\ast$ is called an upper exhauster of the
function $f$ at the point $x$ (respectively, in the sense of Dini
or Hadamard), while the family  $E_\ast$ is called a lower
exhauster of the function $f$ at the point $x$ (respectively, in
the sense of Dini or Hadamard).

It is obvious that an exhauster of the function $f$ at the point
$x$ is also the exhauster of the function $h(g)$ at the origin.


In case when $h$ is continuous at $g$, then the both
representations (\ref{eq1.2}) and (\ref{eq1.3}) are true. In
\cite{cas98} it was shown that if $h$ is Lipschitz then this
function can be written both in forms
\begin{equation}\label{eqs3a1}
h(g)=h_1(g)=\min_{C\in E^*} \max_{v\in C}\langle v,g\rangle \quad
\forall g\in \mathbb{R}^n,
\end{equation}
and
\begin{equation}\label{eqs3a2}
h(g)= h_2(g)=\max_{C\in E_*} \min_{w\in C}\langle w,g\rangle \quad
\forall g\in \mathbb{R}^n,
\end{equation}
where the families of sets $E^\ast$ and $E_\ast$ are totally
bounded. Recall that a family of sets $E$ is totally bounded if
there exists a ball $B$ in $\mathbb{R}^n$ such that $$C\subset
B\quad \forall C\in E.$$ The functions $h_1$ and $h_2$ give p.h.
approximations of the increment of the function $f$ in a
neighborhood of the point $x$. In what follows we will deal with
representations (\ref{eqs3a1}) and (\ref{eqs3a2}).

Exhausters were introduced in
\cite{Demyanov_optimization_99,Demyanov-dem00a}. This notion
brought attention of many researchers
\cite{Demyanov-dros08,Kucuk-Urbanski-Grzybowski,Abbasov_jogo_18,Abbasov_jimo15,Murzobekova09,Luderer03}.
It turned out that unconstrained optimality conditions for the
minimum most organically can be expressed in terms of upper
exhausters (see \cite{Demyanov-dem00a,abbasov_jota_2017,
Demyanov-optimization-2012}). Therefore an upper exhauster was
called proper for the minimization problem and adjoint for the
maximization one.

\begin{theorem}\label{exhausters_th_unconstraint_min extremum_conditions_upper_ex}
If a function $f(x)$ attains a local minimum at a point $x_\ast$
and an upper exhauster $E^\ast$ (in the sense of Dini or Hadamard)
of the function $f(x)$ at the point $x_\ast$ is known, then
$$h(g)=f'(x_\ast,g)=\min_{C\in E^\ast}\max_{v\in C}\langle
v,g\rangle\geq 0 \quad\forall g\in\mathbb{R}^n,$$ what is
equivalent to the condition
\begin{equation}\label{exhausters_th_unconstraint_min_proper_exh_min_0}
 0_n\in C \quad \forall \ C\in E^\ast.
\end{equation}
\end{theorem}

\begin{theorem}\label{exhausters_th_unconstraint_max extremum_conditions_upper_ex}
If a function $f(x)$ attains a local maximum at a point $x^\ast$
and an upper exhauster $E^\ast$ (in the sense of Dini or Hadamard)
of the function $f(x)$ at the point $x^\ast$ is known, then
$$h(g)=f'(x^\ast,g)=\min_{C\in E^\ast}\max_{v\in C}\langle
v,g\rangle\leq 0 \quad\forall g\in\mathbb{R}^n,$$ what is
equivalent to the condition that for every $g\in \mathbb{R}^n$
there exists a set $C(g)\in E^*$ such that
\begin{equation*}\label{exhausters_th_unconstraint_max_adjoint_exh_max_0}
\langle v,g\rangle\geq 0 \quad \forall v\in C(g).
\end{equation*}
\end{theorem}

Symmetric is the situation with a lower exhauster. This family was
called proper for the maximization problem and adjoint for the
minimization one.

\begin{theorem}\label{exhausters_th_unocnstraint_max extremum_conditions}
If a function $f(x)$ attains a local maximum at a point $x^{\ast}$
and a lower exhauster $E_\ast$ (in the sense of Dini or Hadamard)
of the function $f(x)$ at the point $x^{\ast}$ is known, then
$$h(g)=f'(x^\ast,g)=\max_{C\in E_\ast}\min_{v\in C}\langle
v,g\rangle\leq 0 \quad\forall g\in\mathbb{R}^n,$$ is equivalent to
the condition
\begin{equation*}\label{Quasi-Exhastrs_exhausters_proper_exh_max_0}
 0_n\in C \quad \forall \ C\in E_\ast.
\end{equation*}
\end{theorem}

\begin{theorem}\label{exhausters_th_unconstraint_max extremum_conditions_lower_ex}
If a function $f(x)$ attains a local minimum at a point $x_\ast$
and a lower exhauster $E_\ast$ (in the sense of Dini or Hadamard)
of the function $f(x)$ at the point $x_\ast$ is known, then
$$h(g)=f'(x_\ast,g)=\max_{C\in E_\ast}\min_{v\in C}\langle
v,g\rangle\geq 0 \quad\forall g\in\mathbb{R}^n,$$ what is
equivalent to the condition that for every $g\in \mathbb{R}^n$
there exists a set $C(g)\in E_\ast$ such that
\begin{equation}\label{exhausters_th_unconstraint_min_adjoint_exh_min_0}
\langle v,g\rangle\geq 0 \quad \forall v\in C(g).
\end{equation}
\end{theorem}

\section{Problem statement. Conic approximations}

Let functions $f\colon\mathbb{R}^n\to\mathbb{R}$ and
$u\colon\mathbb{R}^n\to\mathbb{R}$ be directionally differentiable
in the sense of Hadamard. Consider the problem
\begin{equation}\label{abbasov_problem_statement_exhaust_on_exhaust_set}
\begin{cases}
f(x)\to \min\\
x\in\Omega
\end{cases}
\end{equation} where $\Omega=\{x\in\mathbb{R}^n\mid u(x)\leq 0\}$.
Since $f$ and $u$ are Hadamard-differentiable both representations
(\ref{eqs3a1}) and (\ref{eqs3a2}) are valid (see Theorems 5.1 and
3.2 in \cite{Demyanov-dr82}) for the derivatives
$f^{\prime}_{H}(x,g)$ and $u^{\prime}_{H}(x,g)$ at the studied
point $x$.

We need conic approximation of the set $\Omega$ in the
neighborhood of the studied point $x$ to derive optimality
conditions for the problem
(\ref{abbasov_problem_statement_exhaust_on_exhaust_set}). Remind
some definitions.

$K_p(x)$ is called cone of possible directions with respect to the
set $\Omega$ at point $x$ if for any $y\in K_p(x)$ there exists
$\overline{\theta}>0$ such that $x+\theta y\in\Omega$ for all
$\theta\in[0,\overline{\theta}]$.

$K_{ad}(x)$ is called cone of admissible directions (or Bouligand
cone) with respect to the set $\Omega$ at the point $x$ if for any
$y\in K_{ad}(x)$ there exists $[\theta_k,y_k]\to [+0,y]$, where
$\theta_k \geq 0$, such that $x+\theta_k y_k\in\Omega$.

Also define cones
$$K_{<}(x)=\left\{y\in\mathbb{R}^n\mid
u^{\prime}_{H}(x,g)<0\right\},\quad
K_{\leq}(x)=\left\{y\in\mathbb{R}^n\mid
u^{\prime}_{H}(x,g)\leq0\right\}.$$ Since the derivative
$u^{\prime}_{H}(x,g)$ is continuous as a function of direction
\cite{Demyanov-dr82}, the cone $K_{<}(x)$ is open (if it is not
empty) and the cone $K_{\leq}(x)$ is closed.

It can be checked easily that
\begin{equation*}\label{abbasov_cones_inclusion_exhaust_on_exhaust_set}
K_{<}(x)\subset K_p(x) \subset K_{ad}(x) \subset
K_{\leq}(x).\end{equation*}

We say that the regularity condition holds at the point $x$ if
$$\operatorname{cl}\left\{K_{<}(x)\right\}=K_{\leq}(x).$$
where $\operatorname{cl}\left\{K_{<}(x)\right\}$ is the closure of
the cone $K_{<}(x)$. This condition provides constructive way for
building the cone $K_{ad}(x)$ which is used in optimal conditions.


\section{Optimality conditions}

We will need the following results (see \cite{Demyanov-dr82}).

\begin{theorem} \label{ch_info_th_exstremum_DH_min} Let $f$ be be directionally differentiable in the sense of Hadamard at the point $x_\ast\in\Omega$. For the point $x_\ast$ to be a local minimizer of $f$ on $\Omega$,
it is necessary that
\begin{equation}\label{ch_info_th_exstremum_DH_min_1}f^{\prime}_H(x_\ast,g)\geq 0 \quad \forall g\in K_{ad}(x_\ast),\end{equation}
where $K_{ad}(x_\ast$) is the Bouligand cone to the set $\Omega$
at the point $x_\ast$.
\end{theorem}

\begin{theorem} \label{ch_info_th_exstremum_DH_max} Let $f$ be be directionally differentiable in the sense of Hadamard at the point $x^\ast\in\Omega$. For the point $x^\ast$ to be a local maximizer of $f$ on $\Omega$,  it is necessary that
\begin{equation}\label{ch_info_th_exstremum_DH_max_1}f^{\prime}_H(x^\ast,g)\leq 0 \quad \forall g\in K_{ad}(x^\ast),\end{equation}
where $K_{ad}(x^\ast$) is the Bouligand cone to the set $\Omega$
at the point $x^\ast$.
\end{theorem}


Now we can formulate necessary optimality conditions for the
problem (\ref{abbasov_problem_statement_exhaust_on_exhaust_set}).
We consider in detail only conditions for the minimum, since
conditions for the maximum can be derived similarly.

First state minimum conditions in terms of (proper) upper
exhauster of the function $f$.

\begin{theorem}\label{abbasov_th_min_cond_up_f_low_u}
Let the regularity condition holds at the point
$x_{\ast}\in\Omega$, families of sets $E^{\ast}(f)$ and
$E_{\ast}(u)$ be an upper and a lower exhausters in the sense of
Hadamard of the functions $f$ and $u$ at the point $x_\ast$
respectively. Then for the point $x_\ast$ to be a local minimum of
the function $f$ on the set $\Omega$ it is necessary that
\begin{equation}\label{abbasov_th_min_cond_up_f_low_u_condition}
\bigcap_{C\in
E_{\ast}(u)}\operatorname{cl}\left\{\mathbb{R}^n\setminus
K^+(C)\right\}\subset \bigcap_{C\in
E^{\ast}(f)}\operatorname{cl}\left\{\mathbb{R}^n\setminus
\left(-K^+(C)\right)\right\},
\end{equation}
where $K(C)=\operatorname{cone}\{C\}$ is the conic hull of set
$C$, $K^+(C)$ is the conjugate cone of $K(C)$.
\end{theorem}

\begin{proof}
Due to Theorem \ref{ch_info_th_exstremum_DH_min} if point $x_\ast$
is a local minimum of the function $f$ on $\Omega$, then
\begin{equation*}\label{abbasov_th_min_cond_up_f_low_u_proof_eq_1}
f^{\prime}_H(x_\ast,g)\geq 0 \quad \forall g\in K_{ad}(x_\ast).
\end{equation*}
Whence considering regularity condition we get
\begin{equation}\label{abbasov_th_min_cond_up_f_low_u_proof_eq_2}
f^{\prime}_H(x_\ast,g)\geq 0 \quad \forall g\in K_{\leq}(x_\ast).
\end{equation}
Using exhauster representation for the directional derivatives in
(\ref{abbasov_th_min_cond_up_f_low_u_proof_eq_2}) we obtain that
inequality $$\min_{C\in E^\ast(f)} \max_{v\in C}\langle v,g\rangle
\geq 0$$ holds for any $g$ such that $$\max_{C\in E_\ast(u)}
\min_{v\in C}\langle v,g\rangle\leq 0.$$ This means that for any
$g\in\mathbb{R}^n$ such that
\begin{equation}\label{abbasov_th_min_cond_up_f_low_u_proof_eq_3}
\min_{v\in C}\langle v,g\rangle\leq 0 \quad \forall C\in
E_\ast(u)
\end{equation}
holds the condition
\begin{equation}\label{abbasov_th_min_cond_up_f_low_u_proof_eq_4}
\max_{v\in C}\langle v,g\rangle
\geq 0 \quad \forall C\in E^\ast(f).
\end{equation}

Inequality (\ref{abbasov_th_min_cond_up_f_low_u_proof_eq_3}) is
equivalent to the fact that for every $C\in E_\ast(u)$ there
exists $v(C)\in C$ such that $\langle v(C),g \rangle\leq 0$.
Consequently $g$ does not lie in the interior of the set
$(\operatorname{cone}\{C\})^+$ for all $C\in E_\ast(u)$. Thus
denoting $K^+(C)$ a conjugate cone of $\operatorname{cone}\{C\}$
we have $g\in\displaystyle\bigcap_{C\in E_{\ast}(u)}
\operatorname{cl}\left\{\mathbb{R}^n\setminus K^+(C)\right\}$.

The same way we can show that any $g$ satisfying inequality
(\ref{abbasov_th_min_cond_up_f_low_u_proof_eq_4}) belongs to the
set $\displaystyle\bigcap_{C\in
E^{\ast}(f)}\operatorname{cl}\left\{\mathbb{R}^n\setminus
\left(-K^+(C)\right)\right\}$ and vice versa.

Therefore (\ref{abbasov_th_min_cond_up_f_low_u_proof_eq_3}) and
(\ref{abbasov_th_min_cond_up_f_low_u_proof_eq_4}) implies
(\ref{abbasov_th_min_cond_up_f_low_u_condition}).
\end{proof}

\begin{corollary} Condition
(\ref{abbasov_th_min_cond_up_f_low_u_condition}) can be
interpreted as follows: for any hyperplane passing through the
origin which nonpositive half-space contains an element from $C\in
E_{\ast}(u)$ for all $C\in E_{\ast}(u)$, there exists an element
from $\widetilde{C}$ which lies in the nonnegative half-space of
this hyperplane for all $\widetilde{C}\in E^{\ast}(f)$.
\end{corollary}

\begin{theorem}\label{abbasov_th_min_cond_up_f_up_u}
Let the regularity condition holds at the point
$x_{\ast}\in\Omega$, families of sets $E^{\ast}(f)$ and
$E^{\ast}(u)$ be upper exhausters in the sense of Hadamard of the
functions $f$ and $u$ at the point $x_\ast$ respectively. Then for
the point $x_\ast$ to be a local minimum of the function $f$ on
the set $\Omega$ it is necessary that
\begin{equation}\label{abbasov_th_min_cond_up_f_up_u_condition}
\bigcup_{C\in E^{\ast}(u)} \left[-K^+({C})\right]\subset
\bigcap_{C\in
E^{\ast}(f)}\operatorname{cl}\left\{\mathbb{R}^n\setminus
\left(-K^+(C)\right)\right\},
\end{equation}
where $K(C)=\operatorname{cone}\{C\}$ is the conic hull of the set
$C$, $K^+(C)$ is the conjugate cone of $K(C)$.
\end{theorem}

\begin{proof}
As in the proof of the previous theorem it can be shown that if
$x_\ast$ is a local minimum of the function $f$ on the set
$\Omega$, then
$$\min_{C\in E^\ast(f)} \max_{v\in C}\langle v,g\rangle \geq 0$$
holds for any $g$ such that $$\min_{C\in E^\ast(u)} \max_{v\in
C}\langle v,g\rangle\leq 0.$$ Therefore for any $g\in\mathbb{R}^n$
such that
\begin{equation}\label{abbasov_th_min_cond_up_f_up_u_proof_eq_1}
\exists \widetilde{C}\in E^\ast(u)\colon \max_{v\in
\widetilde{C}}\langle v,g\rangle\leq 0
\end{equation}
holds the condition
\begin{equation}\label{abbasov_th_min_cond_up_f_up_u_proof_eq_2}
\max_{v\in C}\langle v,g\rangle \geq 0 \quad \forall C\in
E^\ast(f).
\end{equation}
Inequality (\ref{abbasov_th_min_cond_up_f_up_u_proof_eq_1}) is
equivalent to the fact that there exists $\widetilde{C}\in
E^\ast(u)$ such that $\langle v,g \rangle\leq 0$ for all $v\in
\widetilde{C}$. Consequently there exists $\widetilde{C}\in
E^\ast(u)$ such that $g\in -K^+(\widetilde{C})$, where
$K^+(\widetilde{C})$ is a conjugate cone of
$\operatorname{cone}\{\widetilde{C}\}$. 
Whence we conclude that $g\in\displaystyle\bigcup_{C\in
E^\ast(u)}[-K^+(C)]$ is equivalent to condition
(\ref{abbasov_th_min_cond_up_f_up_u_proof_eq_1}). Therefore
(\ref{abbasov_th_min_cond_up_f_up_u_proof_eq_1}) and
(\ref{abbasov_th_min_cond_up_f_up_u_proof_eq_2}) implies
(\ref{abbasov_th_min_cond_up_f_up_u_condition}).
\end{proof}

\begin{corollary} Condition
(\ref{abbasov_th_min_cond_up_f_up_u_condition}) can be interpreted
as follows: for any hyperplane passing through the origin which
nonpositive half-space contains at least one set $C\in
E^{\ast}(u)$, there exists an element from $\widetilde{C}$ which
lies in the nonnegative half-space of this hyperplane for all
$\widetilde{C}\in E^{\ast}(f)$.
\end{corollary}

\begin{remark}If $x_\ast$ is an unconstrained local minimum of $f$ then condition
(\ref{exhausters_th_unconstraint_min_proper_exh_min_0}) holds.
Therefore $$\bigcap_{C\in
E^{\ast}(f)}\operatorname{cl}\left\{\mathbb{R}^n\setminus
\left(-K^+(C)\right)\right\}=\mathbb{R}^n,$$ which implies that
inclusions (\ref{abbasov_th_min_cond_up_f_low_u_condition}) and
(\ref{abbasov_th_min_cond_up_f_up_u_condition}) are always
satisfied. \end{remark}

Now proceed to the minimum conditions in terms of (adjoint) lower
exhauster of the function $f$.

\begin{theorem}\label{abbasov_th_min_cond_low_f_low_u}
Let the regularity condition holds at the point
$x_{\ast}\in\Omega$, families of sets $E_{\ast}(f)$ and
$E_{\ast}(u)$ be lower exhausters in the sense of Hadamard of the
functions $f$ and $u$ at the point $x_\ast$ respectively. Then for
the point $x_\ast$ to be a local minimum of the function $f$ on
the set $\Omega$ it is necessary that
\begin{equation}\label{abbasov_th_min_cond_low_f_low_u_condition}
\bigcap_{C\in
E_{\ast}(u)}\operatorname{cl}\left\{\mathbb{R}^n\setminus
K^+({C})\right\}\subset \bigcup_{C\in E_{\ast}(f)} K^+({C}),
\end{equation}
where $K(C)=\operatorname{cone}\{C\}$ is the conic hull of the set
$C$, $K^+(C)$ is the conjugate cone of $K(C)$.
\end{theorem}

\begin{proof}
As in the proof of Theorem \ref{abbasov_th_min_cond_up_f_low_u} it
can be shown that if $x_\ast$ is a local minimum of the function
$f$ on the set $\Omega$, then
$$\max_{C\in E_\ast(f)} \min_{v\in C}\langle v,g\rangle \geq 0$$
holds for any $g$ such that $$\max_{C\in E_\ast(u)} \min_{v\in
C}\langle v,g\rangle\leq 0.$$ Therefore for any $g\in\mathbb{R}^n$
such that
\begin{equation}\label{abbasov_th_min_cond_low_f_low_u_proof_eq_1}
\min_{v\in C}\langle v,g\rangle\leq 0 \quad \forall C\in E_\ast(u)
\end{equation}
holds the condition
\begin{equation}\label{abbasov_th_min_cond_low_f_low_u_proof_eq_2}
\exists \widehat{C}\in E_{\ast}(f)\colon
\min_{v\in\widehat{C}}\langle v,g\rangle \geq 0.
\end{equation}
Inequality (\ref{abbasov_th_min_cond_low_f_low_u_proof_eq_2}) is
equivalent to the fact that $g\in K^+(\widehat{C})$, where
$K^+(\widehat{C})$ is a conjugate cone of
$\operatorname{cone}(\widehat{C})$. Thus
(\ref{abbasov_th_min_cond_low_f_low_u_proof_eq_1}) and
(\ref{abbasov_th_min_cond_low_f_low_u_proof_eq_2}) implies
(\ref{abbasov_th_min_cond_low_f_low_u_condition}).
\end{proof}

\begin{corollary} Condition
(\ref{abbasov_th_min_cond_low_f_low_u_condition}) can be
interpreted as follows: for any hyperplane passing through the
origin which nonpositive half-space contains an element from $C$
for all $C\in E_{\ast}(u)$, there exists at least one set $C\in
E_{\ast}(f)$ which fully lies in the nonnegative half-space of
this hyperplane.
\end{corollary}

Similarly can be proved the following result.

\begin{theorem}\label{abbasov_th_min_cond_low_f_up_u}
Let the regularity condition holds at the point
$x_{\ast}\in\Omega$, families of sets $E_{\ast}(f)$ and
$E^{\ast}(u)$ be a lower and an upper exhausters in the sense of
Hadamard of the functions $f$ and $u$ at the point $x_\ast$
respectively. Then for the point $x_\ast$ to be a local minimum of
the function $f$ on the set $\Omega$ it is necessary that
\begin{equation}\label{abbasov_th_min_cond_low_f_up_u_condition}
\bigcup_{C\in E^{\ast}(u)} \left[-K^+({C})\right]\subset
\bigcup_{C\in E_{\ast}(f)} K^+({C}),
\end{equation}
where $K(C)=\operatorname{cone}\{C\}$ is the conic hull of the set
$C$, $K^+(C)$ is the conjugate cone of $K(C)$.
\end{theorem}

\begin{corollary} Condition
(\ref{abbasov_th_min_cond_low_f_up_u_condition}) can be
interpreted as follows: for any hyperplane passing through the
origin which nonpositive half-space contains at least one set
$C\in E^{\ast}(u)$, there exists at least one set $C\in
E_{\ast}(f)$ which fully lies in the nonnegative half-space of
this hyperplane.
\end{corollary}

\begin{remark}If $x_\ast$ is an unconstrained local minimum of $f$ then condition
(\ref{exhausters_th_unconstraint_min_adjoint_exh_min_0}) holds.
Therefore $$\bigcup_{C\in E_{\ast}(f)} K^+({C})=\mathbb{R}^n,$$
which implies that inclusions
(\ref{abbasov_th_min_cond_low_f_low_u_condition}) and
(\ref{abbasov_th_min_cond_low_f_up_u_condition}) are always
satisfied. \end{remark}

Analogously via Theorem \ref{ch_info_th_exstremum_DH_max} we can
state and prove conditions for the maximum.

\begin{theorem}\label{abbasov_th_max_cond_up_f_low_u}
Let the regularity condition holds at the point
$x^{\ast}\in\Omega$, families of sets $E_{\ast}(f)$ and
$E_{\ast}(u)$ be lower exhausters in the sense of Hadamard of the
functions $f$ and $u$ at the point $x^\ast$ respectively. Then for
the point $x^\ast$ to be a local maximum of the function $f$ on
the set $\Omega$ it is necessary that
\begin{equation}\label{abbasov_th_max_cond_up_f_low_u_condition}
\bigcap_{C\in
E_{\ast}(u)}\operatorname{cl}\left\{\mathbb{R}^n\setminus
K^+(C)\right\}\subset \bigcap_{C\in
E_{\ast}(f)}\operatorname{cl}\left\{\mathbb{R}^n\setminus
K^+(C)\right\},
\end{equation}
where $K(C)=\operatorname{cone}\{C\}$ is the conic hull of the set
$C$, $K^+(C)$ is the conjugate cone of $K(C)$.
\end{theorem}

\begin{theorem}\label{abbasov_th_max_cond_up_f_up_u}
Let the regularity condition holds at the point
$x^{\ast}\in\Omega$, families of sets $E_{\ast}(f)$ and
$E^{\ast}(u)$ be a lower and an upper exhausters in the sense of
Hadamard of the functions $f$ and $u$ at the point $x^\ast$
respectively. Then for the point $x^\ast$ to be a local maximum of
the function $f$ on the set $\Omega$ it is necessary that
\begin{equation}\label{abbasov_th_max_cond_up_f_up_u_condition}
\bigcup_{C\in E^{\ast}(u)} \left[-K^+({C})\right]\subset
\bigcap_{C\in
E_{\ast}(f)}\operatorname{cl}\left\{\mathbb{R}^n\setminus
K^+(C)\right\},
\end{equation}
where $K(C)=\operatorname{cone}\{C\}$ is the conic hull of the set
$C$, $K^+(C)$ is the conjugate cone of $K(C)$.
\end{theorem}

\begin{theorem}\label{abbasov_th_max_cond_low_f_low_u}
Let the regularity condition holds at the point
$x^{\ast}\in\Omega$, families of sets $E^{\ast}(f)$ and
$E_{\ast}(u)$ be an upper and a lower exhausters in the sense of
Hadamard of the functions $f$ and $u$ at the point $x^\ast$
respectively. Then for the point $x^\ast$ to be a local maximum of
the function $f$ on the set $\Omega$ it is necessary that
\begin{equation}\label{abbasov_th_max_cond_low_f_low_u_condition}
\bigcap_{C\in
E_{\ast}(u)}\operatorname{cl}\left\{\mathbb{R}^n\setminus
K^+({C})\right\}\subset \bigcup_{C\in E^{\ast}(f)}
\left[-K^+({C})\right],
\end{equation}
where $K(C)=\operatorname{cone}\{C\}$ is the conic hull of the set
$C$, $K^+(C)$ is the conjugate cone of $K(C)$.
\end{theorem}

\begin{theorem}\label{abbasov_th_max_cond_low_f_up_u}
Let the regularity condition holds at the point
$x^{\ast}\in\Omega$, families of sets $E^{\ast}(f)$ and
$E^{\ast}(u)$ be upper exhausters in the sense of Hadamard of the
functions $f$ and $u$ at the point $x^\ast$ respectively. Then for
the point $x^\ast$ to be a local maximum of the function $f$ on
the set $\Omega$ it is necessary that
\begin{equation}\label{abbasov_th_max_cond_low_f_up_u_condition}
\bigcup_{C\in E^{\ast}(u)} \left[-K^+({C})\right]\subset
\bigcup_{C\in E^{\ast}(f)} \left[-K^+({C})\right],
\end{equation}
where $K(C)=\operatorname{cone}\{C\}$ is the conic hull of the set
$C$, $K^+(C)$ is the conjugate cone of $K(C)$.
\end{theorem}

\section{An illustrative example}
Consider a function 
$f(x)=|x_1|-|x_2|$ on a set $\Omega=\left\{x\in\mathbb{R}^2\mid
u(x)\leq 0\right\}$ at a point $x_\ast=(0,0)$, where

\begin{figure}[h]
\center{\includegraphics[scale=0.6]{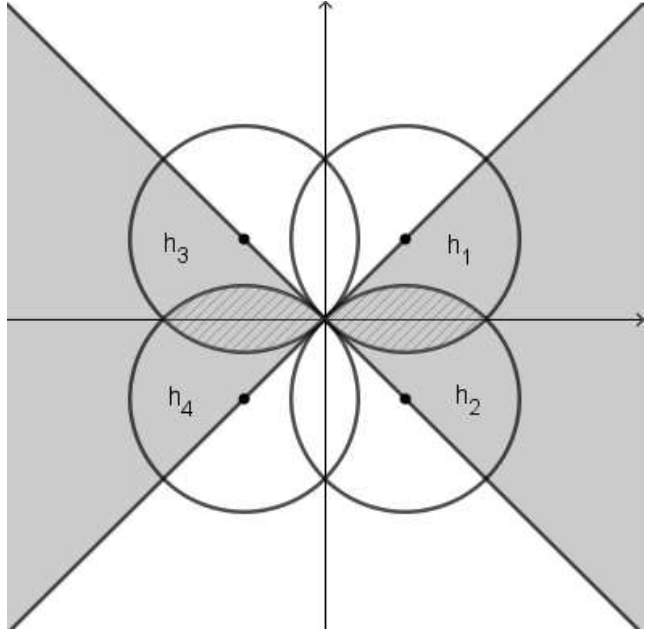}
\caption{The set $\Omega$ (hatched area) and the cone
$K_{ad}(x_\ast)$ (shaded area).}}
\label{example_exh_func_on_exh_set_feasible_set_cone}
\end{figure}

\begin{equation*}
\begin{split}
u(x)&=\min\left\{\max\{
h_1(x),h_2(x)\},\max\{
h_3(x),h_4(x)\}\right\}=\\&=\max\left\{\min\{
h_1(x),h_3(x)\},\max\{ h_2(x),h_4(x)\}\right\},
\end{split}
\end{equation*}
\begin{equation*}
\begin{split}
h_1(x)=\frac{1}{2}\left[(x_1-1)^2+(x_2-1)^2\right]-1, \quad
h_2(x)=\frac{1}{2}\left[(x_1-1)^2+(x_2+1)^2\right]-1,\\
h_3(x)=\frac{1}{2}\left[(x_1+1)^2+(x_2-1)^2\right]-1, \quad
h_4(x)=\frac{1}{2}\left[(x_1+1)^2+(x_2+1)^2\right]-1.
\end{split}
\end{equation*}
%

The functions $f$ and $u$ are Dini-directionally differentiable at
the point $x_\ast$ and Lipschitz, therefore they are directionally
differentiable in the sense of Hadamard \cite{Demyanov-dr82}. It
is also obvious that regularity condition holds at the point
$x_\ast$.

Denote
\begin{equation*}
\begin{split}
C_1=\operatorname{co}\{(1,1);(-1,1)\}, \quad
C_2=\operatorname{co}\{(1,-1);(-1,-1)\},\\
C_3=\operatorname{co}\{(1,1);(1,-1)\}, \quad
C_4=\operatorname{co}\{(-1,1);(-1,-1)\}.
\end{split}
\end{equation*}
Then the following families are exhausters of the functions $f$
and $u$ at the point $x_\ast$:
\begin{equation*}
\begin{split}
E^\ast(f)=E_\ast(u)=\left\{C_1,C_2\right\},\quad
E_\ast(f)=E^\ast(u)=\left\{C_3,C_4\right\}.
\end{split}
\end{equation*}

\begin{figure}[h]
\center{
\begin{minipage}[h]{0.49\linewidth}
\center{\includegraphics[width=1.0 \linewidth]{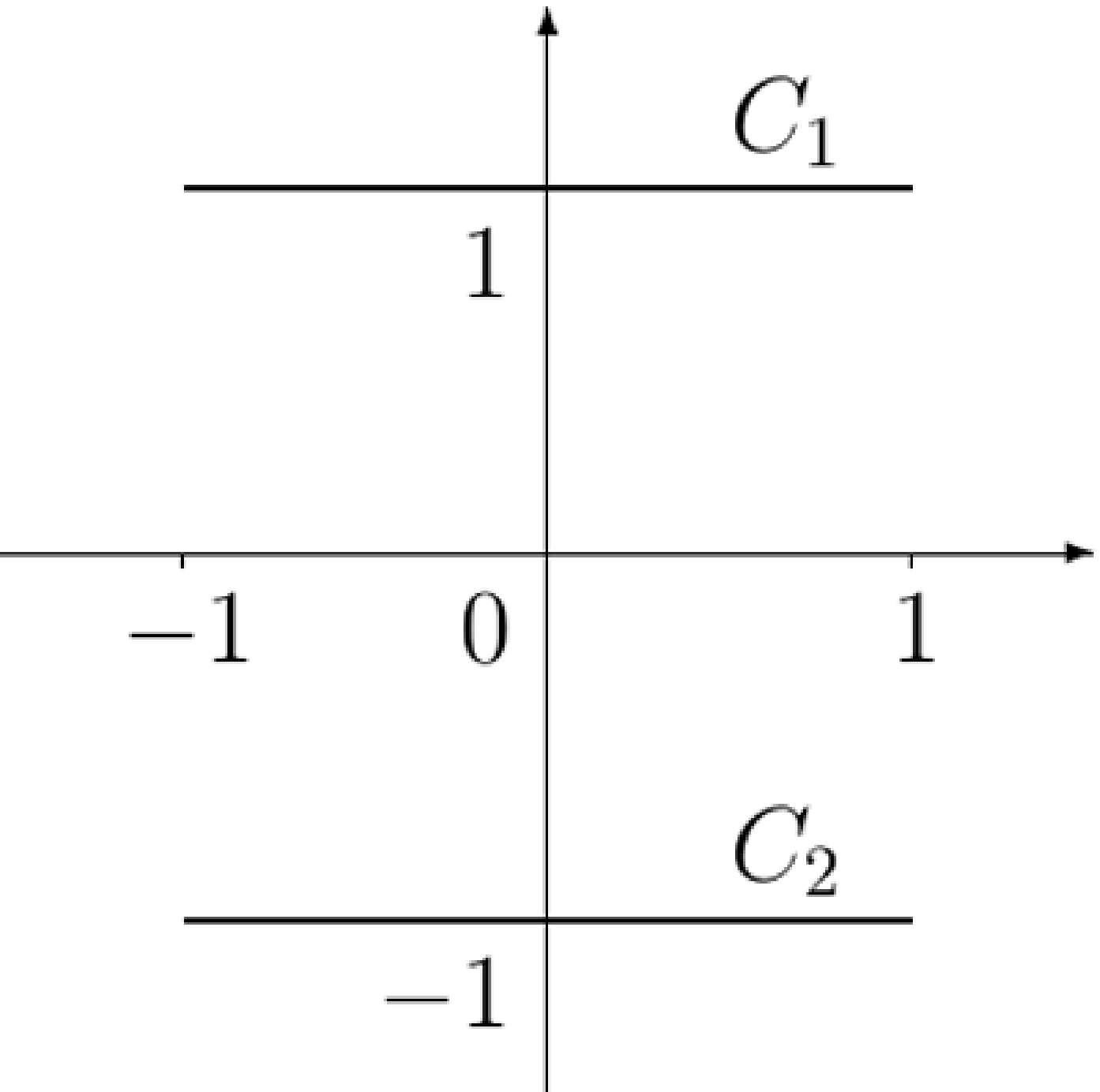} \\
a}
\end{minipage}
\hfill
\begin{minipage}[h]{0.49\linewidth}
\center{\includegraphics[width=1.0 \linewidth]{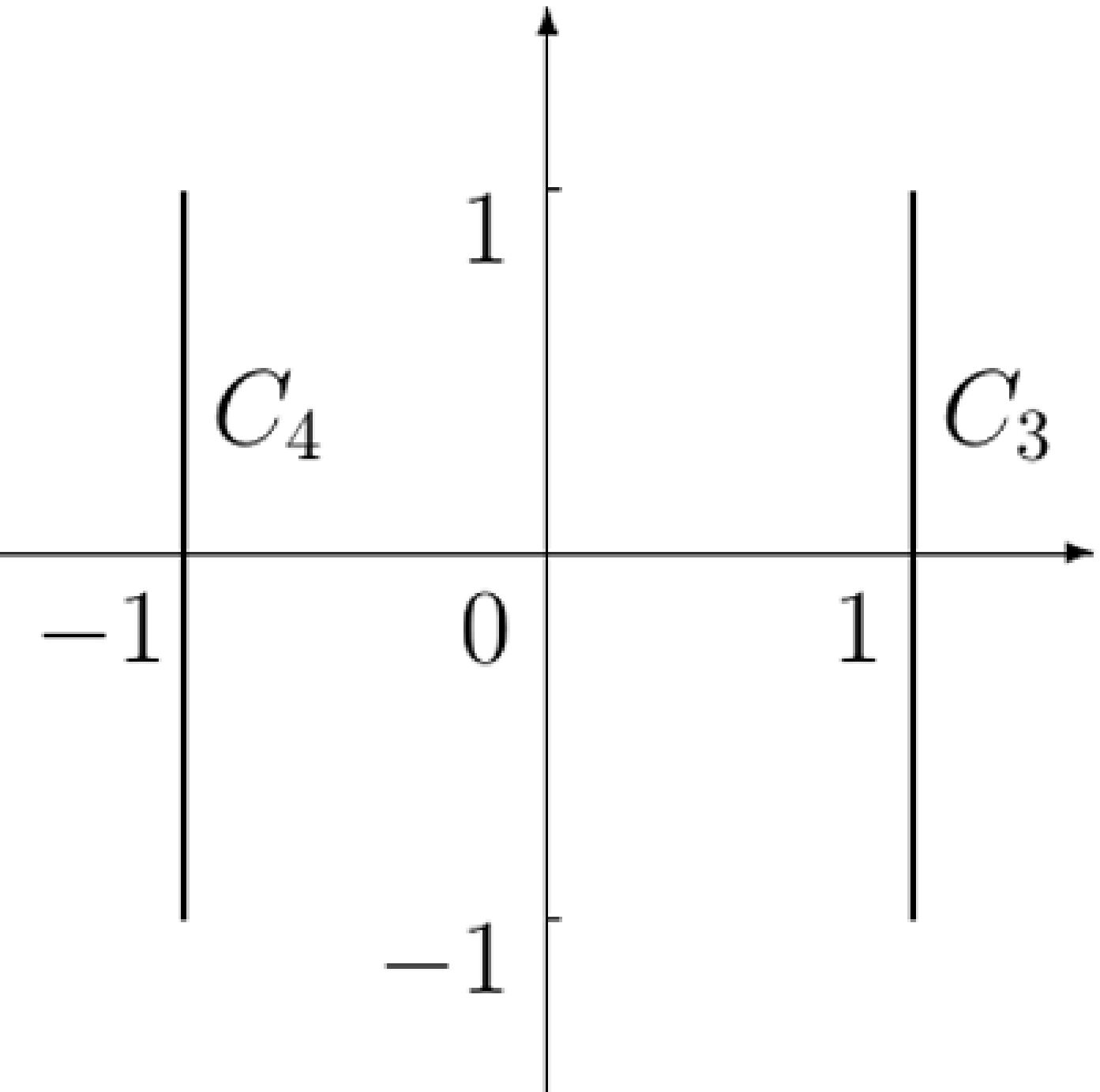} \\
b}
\end{minipage}
\caption{Sets that form exhausters of the functions $f$ and $u$ at
the point
$x_0\ast$.}\label{example_exh_func_on_exh_set_sets_from_exhauters}}
\end{figure}

First check minimality conditions in terms of proper exhauster.
Condition (\ref{abbasov_th_min_cond_up_f_low_u_condition}) from
Theorem \ref{abbasov_th_min_cond_up_f_low_u} is satisfied since
$$\bigcap_{C\in
E_{\ast}(u)}\operatorname{cl}\left\{\mathbb{R}^n\setminus
K^+(C)\right\}=\bigcap_{C\in
E^{\ast}(f)}\operatorname{cl}\left\{\mathbb{R}^n\setminus
\left(-K^+(C)\right)\right\}=K(C_3)\bigcup K(C_4).$$

Condition (\ref{abbasov_th_min_cond_up_f_up_u_condition}) from
Theorem \ref{abbasov_th_min_cond_up_f_up_u} also holds:
$$\bigcup_{C\in E^{\ast}(u)} \left[-K^+({C})\right]=
\bigcap_{C\in
E^{\ast}(f)}\operatorname{cl}\left\{\mathbb{R}^n\setminus
\left(-K^+(C)\right)\right\}=K(C_3)\bigcup K(C_4).$$

Now pass to the minimality conditions in terms of adjoint
exhauster.

We have
$$\bigcap_{C\in
E_{\ast}(u)}\operatorname{cl}\left\{\mathbb{R}^n\setminus
K^+({C})\right\}=\bigcup_{C\in E_{\ast}(f)} K^+({C})=K(C_3)\bigcup
K(C_4),$$ therefore condition
(\ref{abbasov_th_min_cond_low_f_low_u_condition}) from Theorem
\ref{abbasov_th_min_cond_low_f_low_u} is fulfilled.

Condition (\ref{abbasov_th_min_cond_low_f_up_u_condition}) from
Theorem \ref{abbasov_th_min_cond_low_f_up_u} also holds:
$$\bigcup_{C\in E^{\ast}(u)} \left[-K^+({C})\right]=
\bigcup_{C\in E_{\ast}(f)} K^+({C})=K(C_3)\bigcup K(C_4).$$

\section*{Conclusion}
We derived new optimality conditions for an exhausterable function
on an exhausterable set and showed how they can be applied to
practical problems.

It should be noted that constrained optimality conditions for an
exhausterable function on an abstractive theoretical cone were
provided in \cite{Demyanov-dros05}. But these results were only
the first step, as one of the most important problem of
constructing such a cone in specific cases remained open. In the
present paper we considered the case when a feasible set is given
via an exhausterable function, described Bouligand cone in terms
of exhausters of this function and therefore got optimality
conditions in terms of these exhausters. Obtained results can be
applied to a wide class of nondifferentiable optimization
problems.

\section*{Acknowledgements}
The reported study was supported by Russian Science Foundation,
research project  No. 18-71-00006.


\end{document}